\documentclass[11pt]{amsart}

\usepackage[T1]{fontenc}
\usepackage[latin1]{inputenc}
\usepackage{amsmath,amssymb,amsthm}
\usepackage{enumitem}
\usepackage{amsrefs}
\usepackage{hyperref}

\usepackage{calc}
\makeatletter
\renewcommand{\tocsection}[3]{%
  \indentlabel{\ignorespaces#1 \makebox[\widthof{00.}][r]{#2\@ifnotempty{#2}{.}}\quad}#3}
\renewcommand{\tocsubsection}[3]{%
  \indentlabel{\@ifnotempty{#2}{\ignorespaces#1 \makebox[\widthof{00.0.}][r]{\small{#2.}}\,\,  }}\small{#3}}
\makeatother

\theoremstyle{plain}
\newtheorem{thm}{Theorem}[section]
\newtheorem{lem}[thm]{Lemma}

\theoremstyle{definition}

\newtheorem{rem}[thm]{Remark}
\newtheorem{exa}[thm]{Example}

\theoremstyle{remark}

\newcommand{\bB}{\mathbb B}
\newcommand{\bC}{\mathbb C}

\newcommand{\bF}{\mathbb F}

\newcommand{\bR}{\mathbb R}

\newcommand{\cA}{\mathcal A}
\newcommand{\cB}{\mathcal B}

\newcommand{\cH}{\mathcal H}

\newcommand{\cO}{\mathcal O}

\newcommand{\cS}{\mathcal S}

\newcommand{\cU}{\mathcal U}

\DeclareMathOperator{\ran}{ran}
\DeclareMathOperator{\Dim}{dim}

\DeclareMathOperator{\id}{id}

\DeclareMathOperator{\re}{Re}

\DeclareMathOperator{\Tr}{Tr}

\DeclareMathOperator{\spa}{span}

\newcommand{\IF}{\text{ if }}

\newcommand{\qand}{\quad\text{and}\quad}

\newcommand{\ol}[1]{\overline{#1}}
\newcommand{\ep}{\varepsilon}

\newcommand{\Cmin}{{C^*_\textup{min}}}
\newcommand{\ip}[1]{\langle #1 \rangle}
\newcommand{\mt}{\varnothing}

\author{Kenneth R. Davidson}
\address{Pure Mathematics Department, 
University of Waterloo,
Waterloo, ON\ N2L 3G1, 
Canada}
\email{krdavids@uwaterloo.ca}

\author{Michael Hartz}
\address{	Fachrichtung Mathematik\\
Universit\"at des Saarlandes\\
66123 Saarbr\"ucken, 
Germany}
\email{hartz@math.uni-sb.de}
\thanks{KRD was partially funded by NSERC (Canada).
MH was partially supported by the Emmy Noether Program of the German Research Foundation (DFG Grant 466012782).
}

\title{On boundary representations}

\subjclass[2010]{Primary: 46L07; Secondary: 47A20, 46E22, 47L55}
\keywords{Boundary representation, pure state, GNS representation}

\begin{document}
%%%%%%%%%%%%%%%%%%%%%%%

\begin{abstract}
Let $S$ be an operator system sitting in its C*-envelope $\Cmin(S)$.
Starting with a pure state on $S$, let $F$ be the face of state extensions to $\Cmin(S)$.
The dilation theorem of Davidson--Kennedy shows that the GNS representations corresponding to some of the extreme states of $F$ are boundary representations.
We construct an explicit example in which $F$ is an interval and only one of the two extreme points yields a boundary representation.
\end{abstract}

\maketitle

%%%%%%%%%%%%%%%%%%%%%%%
\section{Introduction}
%%%%%%%%%%%%%%%%%%%%%%%

If $S$ is an operator system generating a C*-algebra $\cA$, then an irreducible representation $\pi$ of $\cA$
is called a boundary representation if its restriction to $S$ has a unique completely positive extension to $\cA$, namely $\pi$.
These representations were defined by Arveson \cite{Arveson_subI} and were intended to play the role of points in the Choquet boundary of $S$.
It took a long time for this vision to be realized, depending of work of Hamana \cite{Hamana}, Agler \cite{Agler88}, Muhly--Solel \cite{MuhlySolel}, Dritschel--McCullough 
\cite{DritMcCull}, Arveson \cite{Arveson_choq} and the first author and Kennedy \cite{DavidsonKennedy_boundary}.

In particular, Hamana established the existence of the C*-envelope $\Cmin(S)$, and showed that every other C*-algebra generated by a completely isometric
unital copy of $S$ has a canonical quotient onto $\Cmin(S)$. 

In \cite{DavidsonKennedy_boundary}, the existence of (sufficiently many) boundary representations is established 
by starting with a pure u.c.p.\ map on $S$ (or $M_n(S)$) and dilating it to a maximal u.c.p.\ map, while maintaining purity at each step.
The work of Dritschel--McCullough \cite{DritMcCull} showed that maximal u.c.p.\ maps have the unique extension property,
while the purity ensures that the $*$-representation extending it is irreducible. Thus a boundary representation is obtained.

In particular, one can begin with a pure state $\phi$ on $S$ and consider the face $F$ of all states on $\Cmin(S)$, or any $C^*(S)$, which extend $\phi$. 
The extreme points of this face are pure, and thus generate an irreducible representation by the GNS construction.
The dilation theorem of the previous paragraph shows that some of these irreducible representations are boundary.

This paper began with an investigation of which of these irreducible representations are boundary.
It would have been remarkable had all of them had this property.
It turns out that this is not the case in general.
However the examples which we have found are non-trivial and require some work to establish.
That is the main purpose of this note.

%%%%%%%%%%%%%%%%%%%%%%%
\section{Some Observations}
%%%%%%%%%%%%%%%%%%%%%%%

We refer to Paulsen \cite{Paulsen} for background on operator systems, completely positive maps, and the basics of dilation theory.
We will use $S$ to represent an operator system, and there is no loss in assuming that $S$ is a unital self-adjoint closed subspace of 
the bounded operators, $\cB(H)$, on some Hilbert space $H$.
We will use $H,K$ for Hilbert spaces.

A unital completely positive (u.c.p.) map $\phi$ from an operator system $S$ into $\cB(K)$ is \emph{pure} if 
whenever $\psi$ is a completely positive map of $S$ into $\cB(K)$ such that $0 \le \psi \le \phi$, then $\psi = t\phi$ for some $t \in [0,1]$.
Call $\phi$ a \emph{state} of $S$ if $\Dim K = 1$, and let $\cS(S)$ denote the weak-$*$ compact convex set of all states on $S$.
It is routine to verify that a state $\phi$ is pure if and only if it is an extreme point of $\cS(S)$, which exist by the Krein-Milman theorem.
The same holds for states of $M_n(S)$.
A result of Hopenwasser  \cite{Hopenwasser} shows that a boundary representation of $M_n(S)$
has the form $\id_n \otimes \pi$ where $\pi$ is a boundary representation of $S$.
This was an observation of Kleski \cite{Kleski_email}.
It simplified the argument in \cite{DavidsonKennedy_boundary} that there are sufficiently many boundary representations of $S$
to ensure that the direct sum of all boundary representations is completely isometric on $S$,
and thus yields the C*-envelope as the C*-algebra of its image.

Arveson \cite{Arveson_choqII} calls an operator system \emph{hyper-rigid} if every representation has the unique extension property.
He provides a number of criteria for deciding if an operator system has this property.
Clearly, in this case, in the setting described in the introduction, every extreme point of $F$ gives rise to a boundary representation.
Likewise if $F$ is a singleton, this is the case.

\begin{exa} \label {Ex:cuntz}
Let $\cO_d$ for $2 \le d < \infty$ denote the Cuntz algebra, with standard generators $s_1,\dots, s_d$.
Let $S  = \spa \{ 1, s_1,\dots, s_d, s_1^*,\dots, s_d^* \}$. 
This example is hyper-rigid (see \cite[Corollary 3.4]{Arveson_choqII}), but still has interesting features.

If $\phi \in \cS(S)$, then $(\phi(s_1), \dots, \phi(s_d) ) \subset \ol{\bB_d}$, where $\bB_d$ is the unit ball in $\bC^d$.
It is easy to see that $\phi$ is extreme if and only if $\| (\phi(s_1), \dots, \phi(s_d) ) \|_2 = 1$.
Since the unitary group $\cU_d$ acts on $\cO_d$ by automorphisms, it is enough to consider the point $(1,0,\dots,0)$.
The GNS representation for any extension of this state is the atomic irreducible representation $\sigma_{1,1}$ on $H = \bC \xi \oplus l_2(\bF_d)^{(d-1)}$
described in \cite{DavidsonPitts_fsg}*{Section 3}.
This follows from the analysis in \cite{DavidsonPitts_fsg}, but we will sketch the ideas below.
Therefore, in this case, the face $F$ is a singleton.

Let $\psi$ be any extension of $\phi$ to a state on $\cO_d$,  let $(\pi, \xi)$ be its GNS representation, and set $S_i = \pi(s_i)$
Since $\psi(s_1) = 1$ and $S_1$ is an isometry, we have $S_1 \xi = \xi$. 
Thus $\ran S_i \subset (\bC \xi)^\perp$ for $i \ge 2$. 
But $S_1 (\bC \xi)^\perp \subset (\bC \xi)^\perp$ as well, 
so any word $S_w$ for $w = i_1\dots i_k \in \bF_d^+$ that is not just a sequence of 1's has $S_w\xi \in (\bC \xi)^\perp$.
Hence $\psi(S_w) = 1$ if $w = 1^k$ for $k\ge0$ and $\psi(S_w) = 0$ otherwise.
We can now show that $\xi_i = S_i\xi$ for $2 \le i \le d$ are wandering vectors for $S_1,\dots,S_d$, meaning that $\ip{S_w \xi_i, \xi_i} = 0$ if $w \ne \mt$.
Indeed,
\[
 \ip{S_w \xi_i,\xi_i} = \ip{S_{wi} \xi, S_i \xi} = 
 \begin{cases}
  0 &\IF w \ne iw'\\
  \psi(S_{w'i}) = 0 &\IF w = iw'.
 \end{cases}
\]
It follows that the subspaces $M_i = \overline{\operatorname{span}}\{ S_w\xi_i : w \in \bF_d^+\}$ are invariant subspaces on which 
the $S_i$ act like the left regular representation of $\bF_d^+$ on Fock space.
These subspaces are pairwise orthogonal, and since $\xi$ is cyclic, we get $H = \bC \xi \oplus M_2 \oplus\dots\oplus M_d$.
The representation is now seen to be $\sigma_{1,1}$ as claimed.
It follows that $\psi$ is uniquely determined on $\cO_d$.

More generally, if $\phi:S \to M_n$ is a u.c.p.\ map, then the row matrix $X = \big[ \phi(s_1), \dots, \phi(s_d) \big]$ is a row contraction.
The extreme points are easily seen to be the row co-isometries.
Indeed, these are evidently extreme. Otherwise the polar decomposition $X = |X^*|U$ allows one to write
$X$ as an average of distinct row co-isometries if $|X^*| \ne I_n$.
The row co-isometries dilate to a (unique minimal) row isometry by the Frazho--Bunce--Popescu dilation theorem \cite{Frazho, Bunce, Popescu}; 
and these dilations are in fact row unitary by \cite{DKS}.
So they determine a representations of $\cO_d$, which compresses to u.c.p.\ maps $S \to M_n$.
It follows from \cite{DKS} that it is irreducible (hence boundary) if and only if $C^*(\phi(s_1), \dots, \phi(s_d) ) = M_n$.
This latter issue does not arise in the scalar situation.
\end{exa}

\begin{exa} \label {Ex:affine}
Let $S \subset C(X)$ be a function system; i.e. an operator system in an abelian C*-algebra.
Let $K = \cS(S)$ be its state space. 
By a theorem of Kadison \cite{Kadison}, $S$ is isometrically order isomorphic to $A(K)$, the space of continuous affine functions on $K$
considered as a subspace of $C(K)$.
The state space of $A(K)$ is just the set of evaluations $\ep_x$ at points $x \in K$.
Each state $\ep_x$ extends to the character $\delta_x$ of evaluation at $x$ on $C(K)$.
We consider $\delta_x$ as a probability measure on $K$ representing $x$.
However if $x$ is not extreme, say $x = (x_1+x_2)/2$, then $\ep_x = (\ep_{x_1} + \ep_{x_2})/2$.
Therefore $(\delta_{x_1} + \delta_{x_2})/2$ is another positive extension to $C(K)$; that is, $x$ has other representing measures.
On the other hand, if $x$ is extreme, then $\delta_x$ is the only representing measure for $x$
by a result of Bauer \cite{Phelps}*{Proposition 1.4}.
We conclude that a point in $\cS(A(K))$ is extreme if and only if it has the unique extension property if and only if it is boundary.
\end{exa}

\begin{exa}
Here is a simple example where the face $F$ is not a singleton, but every extreme point of $F$ produces a boundary representation.
It is taken from \cite{CJKLS}*{Example 12}.
Define 
\[
 A = \begin{bmatrix}0&0&0&0\\0&0&0&0\\0&0&1&0\\0&0&0&1 \end{bmatrix}
 \qand
 B = \begin{bmatrix}0&0&2&0\\0&0&0&1\\2&0&0&1\\0&1&1&0 \end{bmatrix} .
\]
Let $S = \spa \{ I_4, A, B \}$.
It is easy to check that $C^*(S) = M_4$, which is simple, so $\Cmin(S) = M_4$.
This C*-algebra has a unique irreducible representation up to unitary equivalence, so it must be a boundary representation for $S$.
The states on $M_4$ are given as $\psi_T(X) = \Tr(XT)$ for $T\ge0$ with $\Tr(T) = 1$;
and the extreme points are the vector states $\psi_{xx^*}(X) = \ip{Xx,x}$ for $\|x\|=1$.

A state $\phi$ on $S$ is determined by $\alpha = \phi(A)$ and $\beta = \phi(B)$.
Let $K$ be the convex set of all pairs $(\alpha,\beta)$.
We can identify $(\alpha,\beta)$ with the point $\alpha + i \beta$ in  $W(A+iB)$, the numerical range of $A+iB$,
as the numerical range is convex by the Toeplitz--Hausdorff theorem.
Clearly $\alpha \in [0,1]$ and $\beta \in \bR$. 
We will see that $(0,0)$ is the unique point in $K \cap \{0\}\times\bR$; and hence it is an extreme point.
But $(0,0)$ has a non-trivial face $F$ of extensions to states on $M_4$.

If $\phi_T(A)=0$, this forces $T = T_1 \oplus O_2$ for $T_1 \in M_2^+$  with $\Tr(T_1) = 1$.
But then $\phi_T(B) = 0$ as well.
This shows that $(0,0)$ is the only point of $K$ on the $y$-axis.
So
\[
 F = \{ \phi_T : T = T_1 \oplus O_2,\ T_1 \in M_2^+,\ \Tr(T_1) = 1 \} .
\]
The extreme points are the vector states $\{ \phi_{xx^*} : x = (x_1,x_2,0,0),\ \|x\| = 1 \}$.
Two vectors $x,y$ give rise to the same state on $M_4$ if and only if $x$ is a unimodular multiple of $y$,
so the extreme points of $F$ can be identified with projective space $P^1(\mathbb{C})$.
\end{exa}

%%%%%%%%%%%%%%%%%%%%%%%
\section{An Example on RKHSs}
%%%%%%%%%%%%%%%%%%%%%%%

The goal of this section is to construct a pure state on an operator system $S \subset \Cmin(S)$ such that
the face $F$ of state extensions to $\Cmin(S)$ is an interval and exactly one of the extreme points
of $F$ yields a boundary representation.
The construction is inspired by some results of Clou\^atre and the second author, see
\cite{CH18} and especially Theorem 6.2 there.
Our operator system $S$ will consist of operators on a reproducing kernel Hilbert space (RKHS)
of analytic functions on the open unit disc $\mathbb{D} \subset \mathbb{C}$.
We recall the necessary basics from the theory of RKHS. Further information
can be found in the books \cite{AM02,PR16}.

For $s \in \mathbb{R}$, let
\begin{equation*}
  \mathcal{H}_s = \Big\{ f(z) = \sum_{n=0}^\infty \widehat{f}(n) z^n \in \mathcal{O}(\mathbb{D}): \|f\|^2 = \sum_{n=0}^\infty (n+1)^{-s} |\widehat{f}(n)|^2 < \infty \Big\}.
\end{equation*}
This scale of spaces contains many important function spaces, such has the Hardy space $H^2$ for $s=0$,
the Bergman space for $s=1$ and the Dirichlet space for $s=-1$.
Each of the spaces $\mathcal{H}_s$ is an RKHS of functions on $\mathbb{D}$, i.e.
for every point $w \in \mathbb{D}$, the functional of evaluation at $w$ is bounded on $\mathcal{H}$. Let
\begin{equation*}
  k_w(z) = K(z,w) = \sum_{n=0}^\infty (n+1)^s (z \overline{w})^n \quad (z,w \in \mathbb{D}).
\end{equation*}
Then $k_w \in \mathcal{H}_s$ and $\langle f,k_w \rangle =f(w)$ for all $f \in \mathcal{H}_s, w \in \mathbb{D}$.
This means that $K$ is the reproducing kernel of $\mathcal{H}_s$.

It is straightforward to check that every polynomial $p \in \mathbb{C}[z]$ defines a bounded linear
multiplication operator $M_p: \mathcal{H}_s \to \mathcal{H}_s, f \mapsto p \cdot f$.
Let $\mathcal{T} \subset B(\mathcal{H}_s)$ denote the unital $C^*$-algebra generated by $M_z$.
Then there is a short exact sequence of $C^*$-algebras
\begin{equation}
  \label{eqn:short_exact}
  0 \longrightarrow \mathcal{K} \longrightarrow \mathcal{T} \longrightarrow C(\mathbb{T}) \to 0,
\end{equation}
where $\mathcal{K}$ denotes the space of compact operators on $\mathcal{H}_s$, the second map is the inclusion and the third map $\rho: \mathcal{T} \to C(\mathbb{T})$
satisfies $\rho(M_p) = p \big|_{\mathbb{T}}$ for $p \in \mathbb{C}[z]$.
Indeed, this follows from the observations that $M_z$ is essentially unitary, so the essential spectrum of $M_z$
equals $\mathbb{T}$ by rotation invariance, and that $\mathcal{T}$ is an irreducible $C^*$-algebra containing a non-zero compact operator.
See for instance \cite[Theorem 4.6]{GHX04} for details.

We are now ready to describe our example.
Let
\begin{equation*}
  \delta: \mathcal{T} \to \mathbb{C}, \quad T \mapsto (\rho(T))(1).
\end{equation*}
Let $\bigvee$ denote the norm closed linear span.
We will work with the operator system
\begin{equation*}
    S = \bigvee\{M_p M_q^*: p, q \in \mathbb{C}[z]\} \subset \mathcal{T}
\end{equation*}
and the state $\varphi = \delta \big|_{S}$, i.e.
\begin{equation*}
  \varphi(M_p M_q^*) = p(1) \overline{q(1)} \quad (p,q \in \mathbb{C}[z]).
\end{equation*}
There is a crucial difference between the cases $s < -1$ and $s \ge -1$.
If $s \ge -1$, then $\delta$ is a boundary representation for the operator
system generated by $\{M_p: p \in \mathbb{C}[z]\}$, and in particular for $S$.
Indeed, for $-1 \le s \le 0$, this statement is part of \cite[Theorem 6.2]{CH18}.
For $s \ge 0$, the operator $M_z$ is a contraction and $\delta(M_z) = 1$,
which easily implies the boundary property by a multiplicative domain argument;
see also \cite[Theorem 3.1.2]{Arveson_subI} and \cite[Theorem 3.4]{CH18} for more general statements.

For $s < -1$, a new phenomenon arises. In this case, $\sum_{n=1}^\infty (n+1)^s < \infty$, and the Cauchy-Schwarz inequality shows that the Taylor coefficients of every function $f \in \mathcal{H}_s$
belong to $\ell^1$; whence $f$ extends to a continuous function on $\overline{\mathbb{D}}$.
The series defining $K$ converges uniformly on $\overline{\mathbb{D}} \times \overline{\mathbb{D}}$,
and the reproducing property of $K$ holds for all points in $\overline{\mathbb{D}}$, so $\mathcal{H}_s$
becomes an RKHS on $\overline{\mathbb{D}}$ in this way.

For $w \in \overline{\mathbb{D}}$, let $\widehat{k}_w = k_w / \|k_w\|$ denote the normalized kernel.
Using the basic fact that $M_p^* k_w = \overline{p(w)} k_w$, we see that
\begin{equation*}
  \langle M_p M_q^* \widehat{k}_1, \widehat{k}_1 \rangle = p(1) \overline{q(1)}
\end{equation*}
for all $p,q \in \mathbb{C}[z]$.
Therefore, in addition to $\delta$, we obtain another extension of $\varphi$ to a state
on $\mathcal{T}$, namely
\begin{equation*}
  \psi: \mathcal{T} \to \mathbb{C}, \quad T \mapsto \langle T \widehat{k}_1, \widehat{k}_1 \rangle.
\end{equation*}
The two extensions $\psi$ and $\delta$ differ,
since $\delta$ annihilates all compact operators, whereas $\psi$ does not.
In particular, $\delta$ is not a boundary representation for $S$ in case $s < -1$.
(For the smaller operator system generated by $\{M_p: p \in \mathbb{C}[z]\}$,
this was already observed in \cite[Theorem 6.2]{CH18}.)

The following result shows that $S$ is an example as announced at the beginning of the section.

\begin{thm}
  \label{thm:H_s}
  Let $s < -1$. With notation as above, we have:
  \begin{enumerate} [label=\normalfont{(\arabic*)}]
    \item $\mathcal{T} = \Cmin(S)$,
    \item $\varphi$ is a pure state of $S$,
    \item the set of states on $\mathcal{T}$ extending $\varphi$ is $\{t \psi + (1 - t) \delta: 0 \le t \le 1 \}$,
    \item the GNS representation of $\psi$ is a boundary representation for $S$, and
    \item $\delta$ is a character that is not a boundary representation for $S$.
  \end{enumerate}
\end{thm}

The proof of Theorem \ref{thm:H_s} occupies the remainder of ths section.
In fact, we will work in slightly greater generality.
Let $(a_n)$ be a sequence of strictly positive numbers with $a_0 = 1$ and $\frac{a_n}{a_{n+1}} \to 1$
and define
\begin{equation*}
  \mathcal{H} = \Big\{ f(z) = \sum_{n=0}^\infty \widehat{f}(n) z^n  \in \mathcal{O}(\mathbb{D}): \|f\|^2 = \sum_{n=0}^\infty \frac{|\widehat{f}(n)|^2}{a_n} < \infty \Big\}.
\end{equation*}
Such a space is called a regular rotationally invariant space on $\mathbb{D}$.
The space $\mathcal{H}$ is an RKHS of functions on $\mathbb{D}$ with reproducing kernel
\begin{equation*}
  k_w(z) = K(z,w) = \sum_{n=0}^\infty a_n (z \overline{w})^n \quad (z,w \in \mathbb{D}).
\end{equation*}
We may define the $C^*$-algebra $\mathcal{T}$ and the operator system $S$ as before.
Again, there is a short exact sequence as in \eqref{eqn:short_exact}.

Assume that $\sum_{n=0}^\infty a_n < \infty$ (equivalently, $K$ is bounded). Then every function in $\mathcal{H}$ extends to a continuous
function on $\overline{\mathbb{D}}$, and $\mathcal{H}$ becomes an RKHS on $\overline{\mathbb{D}}$ in this way.
Every polynomial $p \in \mathbb{C}[z]$ defines a bounded linear
multiplication operator $M_p: \mathcal{H} \to \mathcal{H}, f \mapsto p \cdot f$.
Moreover, we define states
$\varphi$ on $S$ and $\psi,\delta$ on $\mathcal{T}$ as before.

Our computations will become simpler by assuming in addition that $\mathcal{H}$ is a complete Pick space.
Complete Pick spaces are defined in terms of an interpolation property, but for our purposes,
it suffices to use an equivalent description that says that $\mathcal{H}$ is a complete Pick space
if and only if the coefficients $(b_n)_{n=1}^\infty$, defined by the power series identity
\begin{equation*}
  \sum_{n=1}^\infty b_n z^n = 1 - \frac{1}{\sum_{n=0}^\infty a_n z^n},
\end{equation*}
satisfy $b_n \ge 0$ for all $n \ge 1$; see \cite[Theorem 7.33]{AM02}.
The spaces $\mathcal{H}_s$ are complete Pick spaces for all $s \le 0$;  see \cite[Corollary 7.41]{AM02}.

The following lemma in particular establishes properties (1), (4) and (5) in Theorem \ref{thm:H_s}.

\begin{lem}
  Let $\mathcal{H}$ be a regular rotationally invariant complete Pick space on $\mathbb{D}$ with bounded kernel. Then:
  \begin{enumerate}[label=\normalfont{(\alph*)}]
    \item The identity representation of $\mathcal{T}$ is a boundary representation for $S$. In particular,
      $\Cmin(S) = \mathcal{T}$.
    \item The character $\delta$ is not a boundary representation for $S$.
    \item The GNS-representation of $\psi$ is the identity representation, and hence a boundary representation for $S$.
  \end{enumerate}
\end{lem}

\begin{proof}
  (a) Using Arveson's boundary theorem, it was shown in \cite[Theorem 6.2]{CH18} that the identity representation
is a boundary representation for the operator system generated by $\{M_p: p \in \mathbb{C}[z]\}$. In particular, it is a boundary representation for $S$.
It is well known that this implies that $\Cmin(S) = C^*(S) = \mathcal{T}$; this follows from the invariance principle
for boundary representations \cite[Theorem 2.1.2]{Arveson_subI}.

(b) As observed in the discussion preceding Theorem \ref{thm:H_s}, the state $\psi$ is another u.c.p.\ extension of $\delta \big|_S$ to $\mathcal{T}$.

(c) Since $\mathcal{T}$ contains the compact operators, the vector $\widehat{k}_1$ is cyclic for the identity representation
of $\mathcal{T}$, and so the identity representation is the GNS representation of $\psi$.
\end{proof}

Properties (2) and (3) of Theorem \ref{thm:H_s} require a little more work.
We start with the following lemma.

\begin{lem}
  \label{lem:S_basics}
  Let $\mathcal{H}$ be a rotationally invariant complete Pick space on $\mathbb{D}$. 
  Let $P_0 = 11^*$ be the rank 1 projection onto $\bC 1$. Then:
  \begin{enumerate}[label=\normalfont{(\alph*)}]
    \item $\bigvee \{ M_p P_0 M_q^*: p,q \in \mathbb{C}[z]\} = \mathcal{K}$.
    \item $\mathcal{T} = \overline{S + \mathcal{K}}^{\|\cdot\|}$.
    \item $1 - P_0 = \sum_{n=1}^\infty b_n M_{z^n} M_{z^n}^*$, where the sum is SOT-convergent.
    \item $\overline{S}^{w^*} = B(\mathcal{H})$.
  \end{enumerate}
\end{lem}

\begin{proof}
  (a)
  Let $p,q \in \mathbb{C}[z]$.
  We claim that
  \begin{equation*}
    M_p P_0 M_q^* f = \langle f,q \rangle p \quad \text{ for all }  f\in \mathcal{H}.
  \end{equation*}
  Since the linear space of the reproducing kernels is dense in $\cH$, it suffices to check
  this identity on reproducing kernels.
  Let $\lambda, \mu \in \mathbb{D}$. Then
  \begin{align*}
    \langle M_p P_0 M_q^* k_\lambda, k_\mu \rangle &= \overline{q(\lambda)} p(\mu) \langle P_0 k_\lambda, k_\mu \rangle \\&
    = \overline{q(\lambda)} p(\mu) \langle k_\lambda, 1 \rangle \langle 1, k_\mu \rangle = \overline{q(\lambda)} p(\mu) \\&
    = \langle \langle k_\lambda,q \rangle p, k_\mu \rangle,
  \end{align*}
  which yields the desired identity.
  Since the polynomials are dense in $\mathcal{H}$, it follows
  that $\bigvee \{M_p P_0 M_q^*: p, q \in \mathbb{C}[z]\}$ contains all rank one operators,
  and hence all compact operators. The reverse inclusion is obvious.

  (b) We use the short exact sequence \eqref{eqn:short_exact}.
  We have
  \begin{equation*}
    \overline{S + \mathcal{K}}^{\|\cdot\|} / \mathcal{K} \subset \mathcal{T}/\mathcal{K} \cong C(\mathbb{T}),
  \end{equation*}
  and the image of the space on the left inside of $C(\mathbb{T})$ contains $z^n$ for all $n \in \mathbb{Z}$.
  By Fej\'er's theorem, it is dense, hence $\overline{S + \mathcal{K}}^{\|\cdot\|} / \mathcal{K} = \mathcal{T} / \mathcal{K}$, which gives (b).

  (c) Let
  \begin{equation*}
    \Phi(z) =
    \begin{bmatrix}
      \sqrt{b_1} z & \sqrt{b_2} z^2 & \cdots
    \end{bmatrix}.
  \end{equation*}
  Then $K(z,w) (1  - \Phi(z) \Phi(w)^*) = 1$, so a well-known characterization
  of multipliers (see for instance \cite[Theorem 6.28]{PR16}) shows that $\Phi$ is a contractive multiplier.
  Thus for all $N \in \mathbb{N}$,
  \begin{equation*}
    \sum_{n=1}^N b_n M_{z^n} M_{z^n}^* \le M_\Phi M_\Phi^* \le I.
  \end{equation*}
  Since the partial sums are increasing, it follows that the sum converges in SOT.
  To determine the limit, let $\lambda, \mu \in \mathbb{D}$.
  Then
  \begin{align*}
    \Big \langle \sum_{n=1}^\infty b_n M_{z^n} M_{z^n}^* k_\lambda, k_\mu   \Big\rangle 
    &= \sum_{n=1}^\infty b_n \mu^n \overline{\lambda}^n \langle k_\lambda, k_\mu \rangle \\
    &= \Big( 1 - \frac{1}{K(\mu,\lambda)} \Big) K(\mu,\lambda) %\\&
    = K(\mu,\lambda) - 1.
  \end{align*}
  On the other hand
  \begin{equation*}
    \langle (I - P_0) k_\lambda, k_\mu \rangle = K(\mu,\lambda) - 1,
  \end{equation*}
  so the result follows again from the fact that the linear span of the reproducing kernels is dense in $\mathcal{H}$.

  (d) By (c), the weak-$*$ closure of $S$ contains $M_p P_0 M_q^*$ for all $p,q \in \mathbb{C}[z]$. 
  So by (a), it contains all compact operators and hence equals $B(\mathcal{H})$.
\end{proof}

The following lemma is special to RKHS with bounded kernel.

\begin{lem}
  \label{lem:weak_star_extension}
  Let $\mathcal{H}$ be a regular rotationally invariant complete Pick space on  $\mathbb{D}$ with bounded kernel.
  Then every state on $S$ has a unique weak-$*$ continuous extension to a state on $B(\mathcal{H})$.
\end{lem}

\begin{proof}
  Uniqueness is immediate from part (d) of Lemma \ref{lem:S_basics}.
  To establish existence of the weak-$*$ continuous extension, let $\tau$ be any state on $S$.
  By extending $\tau$ to a state on $\mathcal{T}$ and applying the GNS-construction,
  we obtain a $*$-homomorphism $\pi: \mathcal{T} \to B(\mathcal{L})$ and a unit vector $\xi \in \mathcal{L}$ such that
  \begin{equation*}
    \tau(T) = \langle \pi(T) \xi,\xi \rangle \quad \text{ for all } T \in S.
  \end{equation*}
  Let $\rho: \mathcal{T} \to C(\mathbb{T})$ be the quotient map in the short exact sequence \eqref{eqn:short_exact}.
  Then standard results about representations of $C^*$-algebras show that there
  exist a representation $\pi_1$ of $\mathcal{T}$ that is unitarily equivalent to a multiple of the identity representation
  and a representation $\pi_2$ of $C(\mathbb{T})$ such that $\pi = \pi_1 \oplus (\pi_2 \circ \rho)$;
  see the discussion preceding \cite[Theorem 1.3.4]{Arveson76}.
  Thus,
  \begin{equation*}
    \tau(T) = \langle \pi_1(T) \xi_1,\xi_1 \rangle + \langle \pi_2(\rho(T)) \xi_2,\xi_2 \rangle \quad (T \in S),
  \end{equation*}
  where $\|\xi_1\|^2 + \|\xi_2\|^2 = 1$.
  Since $\pi_1$ is unitarily equivalent to a multiple of the identity representation, it is clear that the first summand
  extends to a weak-$*$-continuous positive linear functional of norm $\|\xi_1\|^2$ on $B(\mathcal{H})$.

  It remains to extend the second summand to a weak-$*$ continuous linear functional of norm $\|\xi_2\|^2$ on $B(\mathcal{H})$.
  To this end, it suffices to show the following assertion: If $\sigma$ is a state on $C(\mathbb{T})$, then
  there exists a weak-$*$ continuous state $\omega$ on $B(\mathcal{H})$ such that $(\sigma \circ \rho)\big|_S = \omega\big|_S$.

  Let $\sigma$ be a state on $C(\mathbb{T})$. By the Riesz representation theorem, there exists a Borel probability measure
  $\mu$ on $\mathbb{T}$ such that $\sigma(f) = \int_{\mathbb{T}} f \, d \mu$ for all $f \in C(\mathbb{T})$.
  For $\lambda \in \mathbb{T}$, let $\widehat{k}_\lambda = k_\lambda / \|k_\lambda\|$ and define
  \begin{equation*}
    \omega: B(\mathcal{H}) \to \mathbb{C}, \quad T \mapsto \int_{\mathbb{T}} \langle T \widehat{k}_\lambda, \widehat{k}_\lambda \rangle \, d \mu(\lambda).
  \end{equation*}
  Note that
  \begin{equation*}
    \|\widehat{k}_\lambda - \widehat{k}_\mu\|^2 = 2 - 2 \re \frac{K(\mu,\lambda)}{K(\lambda,\lambda)^{1/2} K(\mu,\mu)^{1/2}},
  \end{equation*}
  so since $K$ is continuous on $\overline{\mathbb{D}} \times \overline{\mathbb{D}}$,
  the function $\lambda \mapsto \widehat{k}_\lambda$ is norm continuous.
  Hence the integrand in the definition of $\omega$ is a continuous function on $\mathbb{T}$.
  In fact,
  \begin{equation*}
    \sup_{\|T\| \le 1} | \langle T \widehat{k}_\lambda, \widehat{k}_\lambda \rangle - \langle T \widehat{k}_\mu, \widehat{k}_\mu \rangle |
    \le 2 \| \widehat{k}_\lambda - \widehat{k}_\mu\|,
  \end{equation*}
  so that the integral $\int_{\mathbb{T}} \langle \cdot \widehat{k}_\lambda, \widehat{k}_\lambda \rangle d \mu(\lambda)$ converges
  in the norm of $B(\mathcal{H})^*$. Consequently, $\omega$ is a weak-$*$ continuous state on $B(\mathcal{H})$.
  If $p,q \in \mathbb{C}[z]$, then
  \begin{equation*}
    \omega(M_p M_q^*) = \int_{\mathbb{T}} \overline{q(\lambda)} p(\lambda) \, d \mu(\lambda)
    = \sigma(\rho(M_p M_q^*)).
  \end{equation*}
  This shows that $\sigma \circ \rho$ agrees with $\omega$ on $S$, as desired.
\end{proof}

\begin{rem}
  In the above proof, the state $\omega$ in general only agrees with $\sigma \circ \rho$ on $S$, not necessarily on $\mathcal{T}$.
  For instance, if $\sigma$ is given by integration against normalized Lebesgue measure on $\mathbb{T}$,
  then $(\sigma \circ \rho)(P_0)= 0$, but $\omega(P_0) = \frac{1}{\|k_1\|^2} \neq 0$.
  In particular, $P_0 \notin S$. This is in contrast to the situation for $H^2$, in which case $S = \mathcal{T}$.
\end{rem}

Since $\varphi = \psi \big|_S$, the following lemma now establishes part (2) of Theorem \ref{thm:H_s}
\begin{lem}
  Let $\mathcal{H}$ be a regular rotationally invariant complete Pick space on $\mathbb{D}$ with bounded kernel.
  If $f \in \mathcal{H}, \|f\| =1$, then the state
  \begin{equation*}
    \varphi_f: S \to \mathbb{C}, \quad T \mapsto \langle T f, f \rangle,
  \end{equation*}
  is pure.
\end{lem}

\begin{proof}
  Let $\varphi_1,\varphi_2$ be states on $S$ with $\varphi_f = \frac{\varphi_1 + \varphi_2}{2}$.
  By Lemma \ref{lem:weak_star_extension}, we may extend $\varphi_1,\varphi_2$ to weak-$*$ continuous states $\psi_1,\psi_2$ on $B(\mathcal{H})$.
  Then
  \begin{equation*}
    \langle T f,f \rangle = \frac{\psi_1(T) + \psi_2(T)}{2}
  \end{equation*}
  for all $T \in S$, and hence for all $T \in B(\mathcal{H})$ by Lemma \ref{lem:S_basics} (d).
  Since vector states on $B(\mathcal{H})$ are pure, $\psi_1 = \psi_2$, and in particular $\varphi_1 = \varphi_2$.
\end{proof}

Finally, we obtain part (3) of Theorem \ref{thm:H_s}.

\begin{lem}
  Let $\mathcal{H}$ be a regular rotationally invariant complete Pick space on $\mathbb{D}$ with bounded kernel.
  Let $\varphi,\psi,\delta$ be the states defined before Theorem $\ref{thm:H_s}$.
  Then the face of all states of $\mathcal{T}$ extending $\varphi$ is a line segment equal to
 \[
   \{t \psi + (1 - t) \delta: 0 \le t \le 1\}.
 \]
\end{lem}

\begin{proof}
  By the discussion preceding Theorem \ref{thm:H_s}, each $t \psi + ( 1 - t) \delta$ is a state extending $\varphi$.
  Conversely, let $\tau$ be any state of $\mathcal{T}$ extending $\varphi$.
  Then
  \begin{equation*}
    \tau(M_p M_p^*) = \varphi(M_p M_p^*) = |p(1)|^2 = |\varphi(M_p)|^2 = |\tau(M_p)|^2.
  \end{equation*}
  By multiplicative domains (see for instance \cite[Theorem 3.18]{Paulsen}), we have
  \begin{equation*}
    \tau(M_p P_0 M_q^*) = \tau(M_p) \tau(P_0) \tau(M_q^*)
    = \varphi(M_p) \tau(P_0) \varphi(M_q^*)
  \end{equation*}
  for all $p,q \in \mathbb{C}[z]$.
  Lemma \ref{lem:S_basics} (a) and (b) therefore show that $\tau$ is uniquely determined by the value $\tau(P_0)$.
  Let $\alpha = \frac{1}{\|k_1\|^2}$. Then part (c) of Lemma \ref{lem:S_basics} shows that
  \begin{equation*}
    P_0 \le I - \sum_{n=1}^N b_n M_{z^n} M_{z^n}^*
  \end{equation*}
  for all $N \in \mathbb{N}$, so
  \begin{equation*}
    0 \le \tau(P_0) \le 1 - \sum_{n=1}^\infty b_n = \frac{1}{K(1,1)} =  \alpha.
  \end{equation*}

  On the other hand,
  \begin{equation*}
    \psi(P_0)= \langle P_0 \widehat{k}_1, \widehat{k}_1 \rangle = \frac{1}{\|k_1\|^2} = \alpha
  \end{equation*}
  and $\delta(P_0) = 0$. Therefore,
  \begin{equation*}
    \tau = \frac{\tau(P_0)}{\alpha} \psi + \big( 1 - \frac{\tau(P_0)}{\alpha} \big) \delta
  \end{equation*}
  is a convex combination of $\psi$ and $\delta$.
\end{proof}

\section*{Acknowledgements} 
The first author thanks the Universit\"at des Saarlandes, in particular the second author, for their hospitality
during a recent visit in which much of this paper was written.

\end{document}